\documentclass[12pt,reqno]{amsart}
\usepackage[left=1in,right=1in,top=1in,bottom=1in]{geometry}
\usepackage{amssymb}
\usepackage{amsmath}
\usepackage{mathrsfs}
\usepackage{amsfonts}
\usepackage{hyperref}
\usepackage{xcolor}
\usepackage{bbm}
\usepackage{tikz}
\usepackage{array}
\usepackage{multirow}
\usepackage{tikz}
\usepackage{ytableau}
\usepackage{thmtools}
\usepackage{thm-restate}
\usepackage{varwidth}
\usepackage{comment}
\usepackage{mathtools}
\usepackage[mathscr]{euscript}
\usepackage{bbm}
\usepackage{multicol}
\usepackage{enumerate}
\usepackage{graphicx} % Required for inserting images
\theoremstyle{definition}
\newtheorem{thm}{Theorem}[section]
\newtheorem{lem}[thm]{Lemma}
\newtheorem*{lem*}{Lemma}
\newtheorem*{thm*}{Theorem}
\newtheorem*{cor*}{Corollary}
\newtheorem{prop}[thm]{Proposition}

\newtheorem{cor}[thm]{Corollary}

\newtheorem{defn}[thm]{Definition}
\newtheorem*{remark*}{Remark}

\newtheorem{cor/defn}[thm]
{Corollary/Definition}

\DeclareMathOperator{\Par}{\mathbb{Y}}

\DeclareMathOperator{\Hom}{\mathrm{Hom}}
\DeclareMathOperator{\fX}{\mathfrak{X}}

\usepackage[
    backend=bibtex,
    maxnames=5,
    style=alphabetic
  ]{biblatex}
\title{Wreath Generalization of Littlewood Reciprocity}

\author{Milo Bechtloff Weising}
\address{Department of Mathematics (0123),
460 McBryde Hall, Virginia Tech,
225 Stanger Street,
Blacksburg, VA 24061-1026}
\email{milojbw@vt.edu}

\date{\today}

\begin{document}

\maketitle 

\begin{abstract}
    Given any $m$-dimensional complex representation $\eta$ of a finite group $G$ and any highest weight representation $V^{\lambda}$ of $\mathrm{GL}_{nm}(\mathbb{C})$ we may define an action of $G^n \rtimes \mathfrak{S}_n$ on $V^{\lambda}$ using the embedding $\mathrm{GL}_{m}(\mathbb{C})^n \rtimes \mathfrak{S}_n \leq \mathrm{GL}_{nm}(\mathbb{C})$ and $\eta: G \rightarrow \mathrm{GL}_m(\mathbb{C})$. We derive a branching rule for the multiplicities of irreducible $G^n \rtimes \mathfrak{S}_n$ representations in $V^{\lambda}.$ The formula generalizes Littlewood's reciprocity rule for branching between $\mathrm{GL}_n(\mathbb{C})$ and the symmetric group of permutation matrices $\mathfrak{S}_n \leq \mathrm{GL}_n(\mathbb{C}).$
  
\end{abstract}

\section{Introduction}

The \textit{\textbf{restriction problem}} is the long-standing open problem in algebraic combinatorics which asks for a combinatorial formula for the \textit{\textbf{branching coefficients}} from $\mathrm{GL}_n(\mathbb{C})$ to $\mathfrak{S}_n$ embedded as the group of permutation matrices. More concretely, given a \textit{\textbf{highest weight representation}} $V^{\lambda}$ of $\mathrm{GL}_n(\mathbb{C})$ we may restrict to $\mathfrak{S}_n$ and try to describe the expansion of $\mathrm{Res}^{\mathrm{GL}_n(\mathbb{C})}_{\mathfrak{S}_n} V^{\lambda}$ into the irreducible \textit{\textbf{Specht modules}} $S^{\mu}$ of $\mathfrak{S}_n.$ We refer the reader to Orellana--Zabrocki's paper \cite{OZ21} for an overview of the current state of this problem. Littlewood \cite{Littlewood58} made significant progress towards this goal by establishing an algebraic formula for these coefficients involving \textit{\textbf{Schur symmetric functions}} $s_{\lambda}(X)$, the \textit{\textbf{Hall pairing}} $\langle -, - \rangle$, and \textit{\textbf{plethystic substitutions}}. The formula may be described as follows:

\begin{thm*}[Littlewood]\cite{Littlewood58}
    For $\lambda,\mu \in \Par$ with $\ell(\lambda) \leq |\mu| = n$, 
    $$\mathrm{dim}_{\mathbb{C}} \mathrm{Hom}_{\mathfrak{S}_{n}}\left( S^{\mu}, \mathrm{Res}^{\mathrm{GL}_n(\mathbb{C})}_{\mathfrak{S}_n} V^{\lambda} \right) = \left\langle s_{\mu}\left( \sum_{k\geq 0} h_k \right), ~ s_{\lambda} \right\rangle.$$
\end{thm*}

We will generalize this formula to compute the branching coefficients from $\mathrm{GL}_n(\mathbb{C})$ to $G^n \rtimes \mathfrak{S}_n$ where $m \geq 1$ is arbitrary along any representation $\eta: G \rightarrow \mathrm{GL}_m(\mathbb{C})$ which we assume to be unitary without loss of generality. The irreducible representations of $G^n \rtimes \mathfrak{S}_n$, $W_{\rho}$, are indexed by partition valued functions $\rho:G_* \rightarrow \Par$ with $|\rho| = \sum_{c \in G_*} |\rho(c)| = n$ where $G_*$ is the set of equivalence classes of irreducible complex representations of $G$. For $n\geq 0$, we will write $\eta^{(n)}:G^n \rtimes \mathfrak{S}_n \rightarrow \mathrm{U}(m)^{n} \rtimes \mathfrak{S}_n$ for the natural lift of $\eta$ to $G^n \rtimes \mathfrak{S}_n.$ The formula is as follows:

\begin{thm*}
For $\rho: G_* \rightarrow \Par$ with $|\rho| = n$, an $m$-dimensional unitary representation $\eta: G \rightarrow \mathrm{U}(m)$, and $\lambda \in \Par$ with $\ell(\lambda) \leq n m ,$
    $$\mathrm{dim}_{\mathbb{C}} \mathrm{Hom}_{G^{n} \rtimes \mathfrak{S}_n}\left( W_{\rho}, \eta^{(n)}_*\mathrm{Res}^{\mathrm{GL}_{nm}(\mathbb{C})}_{\mathrm{U}(m)^{n} \rtimes \mathfrak{S}_n} V^{\lambda} \right)= \left\langle \prod_{\gamma \in G_*} s_{\rho(\gamma)}\left(\sum_{\mu \in \Par} \dim \Hom_{G}(\gamma, \mathbb{S}_{\mu}(\eta)) s_{\mu}\right), ~ s_{\lambda} \right\rangle.$$
\end{thm*}

Let $m \geq 1$, fix $\zeta_m:= e^{\frac{2 \pi i}{m}}$ to be a primitive complex $m$-th root of unity, and let $G = \mu_m$ be the cyclic group of $m$-th roots of unity in $\mathbb{C}^*$. We may think of irreducible characters of $\mu_m$ as being in natural bijection with $\mu_m$ itself, i.e., $(\mu_m)_* \equiv \mu_m.$ We apply Theorem \ref{main theorem} to the case when $G = \mu_m$ to obtain Corollary \ref{cyclic cor}:

\begin{cor*}
For $\rho: \mu_m \rightarrow \Par$ and $\lambda \in \Par$ with $\ell(\lambda) \leq |\rho|= n,$
    $$\mathrm{dim}_{\mathbb{C}} \mathrm{Hom}_{(\mu_m)^{n} \rtimes \mathfrak{S}_n}\left( W_{\rho}, \mathrm{Res}^{\mathrm{GL}_n(\mathbb{C})}_{(\mu_m)^{n} \rtimes \mathfrak{S}_n} V^{\lambda} \right) = \left\langle \prod_{0\leq j \leq m-1} s_{\rho(\zeta_m^j)}\left( \sum_{k \geq 0} h_{km+j}\right), ~ s_{\lambda} \right\rangle.$$
\end{cor*}

The proof of Theorem \ref{main theorem} follows the following steps. First, we describe the \textit{\textbf{Frobenius characteristics}} of the $G^{n} \rtimes \mathfrak{S}_n$-modules $\eta^{(n)}_*\mathrm{Res}^{\mathrm{GL}_{nm}(\mathbb{C})}_{\mathrm{U}(m)^{n} \rtimes \mathfrak{S}_n} V^{\lambda}$ by explicitly computing the characteristic polynomials of $\eta(\gamma) \in \mathrm{GL}_{nm}(\mathbb{C}).$ Next, we utilize the interplay between symmetric functions $\Lambda$ and \textit{\textbf{wreath symmetric functions}} $\bigotimes_{c \in G^*} \Lambda(X_{c})$ in order to algebraically describe the evaluations of Schur functions $s_{\lambda}$ at the eigenvalues of the matrices $\eta(\gamma) \in \mathrm{GL}_{nm}(\mathbb{C}).$ There is a straightforward generating function which controls these substitutions, namely $\Omega\left( \sum_{c \in G^*} \zeta_c^{-1} X_{c} \Omega_{c,\eta}(Y) \right)$ where $\Omega_{c,\eta}(Y):= \prod_{i\geq 1} \det( \mathrm{Id}_m - y_i\eta(\gamma))^{-1}$ for any $\gamma \in c$, which we use to describe the Frobenius characteristics of the restrictions $\eta^{(n)}_*\mathrm{Res}^{\mathrm{GL}_{nm}(\mathbb{C})}_{\mathrm{U}(m)^{n} \rtimes \mathfrak{S}_n} V^{\lambda}$. Lastly, by using the reproducing kernel property for the pairing on wreath symmetric functions $\bigotimes_{c \in G^*} \Lambda(X_{c})$ we arrive at the main result.

\section{Set-up}

In this section, we will establish some necessary notation. 

\begin{defn}
    For $n \geq 1$, define $\mathrm{GL}_n(\mathbb{C})$ as the group of $n\times n$ invertible complex matrices and define $\mathfrak{S}_n$ as the symmetric group of permutations of $\{1,\ldots,n\}.$ We write $\mathrm{U}(m)$ for the group of $m\times m $ unitary matrices. For the remainder of this paper fix some finite $G$ and a unitary representation $\eta: G \rightarrow \mathrm{U}(m)$. For $n \geq 1$, define $\mathfrak{S}_n(\mathrm{U}(m)):= \mathrm{U}(m)^n\rtimes \mathfrak{S}_n$ where $\mathfrak{S}_n$ acts on $\mathrm{U}(m)^n$ by coordinate permutations. We will write elements of $\mathfrak{S}_n(\mathrm{U}(m))$ as $(g_1,\ldots,g_n; \sigma)$ where $g_1,\ldots, g_n \in \mathrm{U}(m)$ and $\sigma \in \mathfrak{S}_n.$ We embed $\mathfrak{S}_n(\mathrm{U}(m))$ into $\mathrm{GL}_{nm}(\mathbb{C})$ as the group of block matrices 
    $$M = \begin{pmatrix}
        A_{11} & \cdots & A_{1n} \\
        \vdots & \ddots & \vdots \\
        A_{n1} & \ldots & A_{nn} \\
    \end{pmatrix}$$ where each $A_{ij}$ is a matrix with size $m\times m$ with exactly one nonzero block in each row and column whose nonzero blocks $A_{ij} \in \mathrm{U}(m) \subset \mathrm{GL}_m(\mathbb{C})$. Analogously, we define the group $\mathfrak{S}_n(G):= G^n \rtimes \mathfrak{S}_n$ and we write $\eta^{(n)}: \mathfrak{S}_n(G) \rightarrow \mathfrak{S}_n(\mathrm{U}(m)) \subset \mathrm{GL}_{nm}(\mathbb{C})$ for the map $\eta^{(n)}(g_1,\ldots,g_n;\sigma):= (\eta(g_1),\ldots, \eta(g_n);\sigma).$ To be consistent with trivial cases, we set $\mathrm{GL}_0(\mathbb{C}) = \mathfrak{S}_0 = \mathfrak{S}_0(\mathrm{U}(m)) = \mathfrak{S}_0(G) := \{1\}$ to all be the trivial group.
\end{defn}

Here we define the \textit{\textbf{ring of symmetric functions}}.

\begin{defn}
    Define $\Par$ as the set of non-negative integer\textit{\textbf{ partitions}}. Given $\lambda = (\lambda_1,\ldots,\lambda_{\ell}) \in \Par$ we write $|\lambda|:= \lambda_1+\ldots + \lambda_{\ell}$, $\ell(\lambda):= \ell,$ and for $r \geq 1$, $m_{r}(\lambda):= \#( 1\leq i \leq \ell | \lambda_i = r).$ For $\lambda \in \Par$, define $z_{\lambda}:= \prod_{r\geq 1} m_r(\lambda)! r^{m_r(\lambda)}.$ We will write $\Lambda(X)$ for the ring of symmetric functions over $\mathbb{C}$ in the variable set $X = x_1+x_2+\ldots ~.$ We write $p_{\lambda}(X),h_{\lambda}(X)$ and $s_{\lambda}(X)$ for the \textit{\textbf{power sum}}, \textit{\textbf{complete homogeneous}}, and \textit{\textbf{Schur symmetric function}} bases of $\Lambda(X)$ respectively. Define the hermitian \textit{\textbf{Hall pairing}} $\langle -,- \rangle: \Lambda \times \Lambda \rightarrow \mathbb{C}$ by

    $$ \left\langle \sum_{\lambda \in \Par} a_{\lambda} p_{\lambda}(X) , \sum_{\lambda \in \Par} b_{\lambda} p_{\lambda}(X) \right \rangle := \sum_{\lambda \in \Par} z_{\lambda} a_{\lambda}\overline{b_{\lambda}}.$$ Define the \textit{\textbf{plethystic exponential}} as $\Omega(X):= \sum_{n \geq 0} h_n(X).$ 
\end{defn}

Throughout this paper, we will use \textit{\textbf{plethystic substitution}} of symmetric functions. We will define plethystic substitution as follows:

\begin{defn}
    Let $t_1,t_2,\ldots$ be a set of commuting free variables. For a formal series $A \in \mathbb{C}[[t_1,t_2,\ldots ]]$, say $A= \sum_{\alpha} c_{\alpha} t^{\alpha}$ and $n \geq 1$ define 
    $p_n(A) = \sum_{\alpha} c_{\alpha} t^{n\alpha}.$ Now define the plethystic substitution $F \mapsto F(A)$ as the $\mathbb{C}$-algebra homomorphism $\Lambda \rightarrow \mathbb{C}[[t_1,t_2,\ldots ]]$ uniquely determined by $p_n \mapsto p_n(A).$
\end{defn}

We will need to work with a generalization of the ring of symmetric functions with multiple variable sets. Before, however, we need to establish some notation.

\begin{defn}
    Write $G^*$ for the set of conjugacy classes of $G$ and for each $c \in G^*$ write $\zeta_c := \frac{|G|}{|c|}$ for size of the centralizer of the conjugacy class $c.$ Given $\gamma = (g_1,\ldots,g_n;\sigma) \in \mathfrak{S}_n(G)$ and a cycle $i_1,\ldots,i_{\ell}$ of $\sigma$, the element $g_{i_1}\cdots g_{i_{\ell}} \in G$ is called the \textit{\textbf{cycle product}} of $\gamma$ corresponding to $i_1,\ldots,i_{\ell}$ and $\ell$ is referred to as the length of the cycle product. Let $\Phi$ denote the set of functions $\rho: G^* \rightarrow \Par$ and for $n \geq 0$ we let $\Phi_n \subset \Phi$ denote the set of those $\rho$ with $|\rho| = n$ where $|\rho|:= \sum_{c \in G^*} |\rho(c)|.$ Given $\rho \in \Phi_n$, define the conjugacy class $C_{\rho} \subset \mathfrak{S}_n(G)$ as the set of those elements $\gamma \in \mathfrak{S}_n(G)$ such that for all $c \in G^*$ and $\ell \geq 1,$ $\gamma$ has $m_{\ell}(\rho(c))$ cycle products of length $\ell$ belonging to the conjugacy class $c.$ For $\rho \in \Phi,$ define $Z_{\rho}:= \prod_{c \in G^*} z_{\rho(\zeta)} \zeta_c^{\ell(\rho(\zeta))}.$
\end{defn}

Now we define the \textit{\textbf{wreath symmetric functions}} which help to describe the representation theory of the wreath product groups $\mathfrak{S}_n(G)$ as $n$ varies. We refer the reader to Appendix B in Macdonald's book \cite{Macdonald} as well as Ingram--Jing--Stitzinger's paper \cite{IJS09} for a more complete overview of this topic.

\begin{defn}
    Define the ring $R:= \bigotimes_{c \in G^*} \Lambda(X_{c})$ where for $c \in G^*$, $X_{c} = x_{c,1}+x_{c,2}+\ldots$ is a corresponding variable set. We will often write $f(\fX) = f(X_{c}| c \in G^*)$ for elements $f \in R.$ 
    For $\rho \in \Phi$, define 
    $P_{\rho}(\fX):= \prod_{c \in G^*} p_{\rho(c)}(X_{c}) .$ Let $G_*$ denote the set of equivalence classes of irreducible complex representations of $G.$ We will write $\chi^{\gamma}(c)$ for the value of the irreducible character $\chi^{\gamma}$ corresponding to $\gamma \in G_*$ at the conjugacy class $c \in G^*.$ Define $\Psi$ as the set of functions $\rho:G_* \rightarrow \Par$ and for $n \geq 0$ let $\Psi_n$ denote the set of those $\rho$ with $|\rho| = n$ where $|\rho|:= \sum_{\gamma \in G^*} |\rho(\gamma)|.$ For $\rho \in \Psi$ define  
    $S_{\rho}(\fX):= \prod_{\gamma \in G_*} s_{\rho(\gamma)}(\varphi_{\gamma}(\fX))$ where for $\gamma \in G_*$, 
    $\varphi_{\gamma}(\fX):= \sum_{c\in G^*}  \zeta_c^{-1}\chi^{\gamma}(c) X_c$ and $s_{\rho(\gamma)}(\varphi_{\gamma}(\fX))$ denotes a plethystic substitution. 
    Define the hermitian pairing $(-,-):R\times R \rightarrow \mathbb{C}$ by 
    $$\left( \sum_{\rho \in \Phi} a_{\rho} P_{\rho}(\fX) , \sum_{\rho \in \Phi} b_{\rho} P_{\rho}(\fX) \right):= \sum_{\rho \in \Phi} Z_{\rho} a_{\rho} \overline{b_{\rho}}$$ Lastly, for $f(\fX) = \sum_{\rho } a_{\rho} P_{\rho}(\fX)$ define $\overline{f(\fX)}:= \sum_{\rho } \overline{a_{\rho}} P_{\rho}(\fX).$
\end{defn}

\section{Branching coefficients}

We begin by defining the branching coefficients which we will be studying for the remainder of this paper. 

\begin{defn}
    Let $n \geq 0$.
    For $\lambda \in \Par$ with $\ell(\lambda)\leq n$, define $V^{\lambda}$ to be the \textit{\textbf{highest weight representation}} of $\mathrm{GL}_n(\mathbb{C})$ corresponding to $\lambda.$ For $\rho \in \Psi_n$, we will write $W_{\rho}$ for the corresponding irreducible representation of $\mathfrak{S}_n(G)$ given explicitly as 
    $$W_{\rho}:= \mathrm{Ind}_{\prod_{\gamma \in G_*} \mathfrak{S}_{|\rho(\gamma)|}(G) }^{\mathfrak{S}_n(G)} \bigotimes_{\gamma \in G_*} S^{\rho(\gamma)}\otimes \gamma^{\otimes |\rho(\gamma)|}$$ where $S^{\rho(\gamma)}$ are the \textit{\textbf{Specht modules}} of $\mathfrak{S}_{|\rho(\gamma)|}$ corresponding to the partitions $\rho(\gamma)$.
    For $\ell(\lambda) \leq nm$ and $\rho \in \Psi_n$, define the \textit{\textbf{branching coefficient}} 
    $$d_{\rho,\lambda}^{\eta}:= \mathrm{dim}_{\mathbb{C}} \mathrm{Hom}_{\mathfrak{S}_{n}(G)}\left( W_{\rho}, \eta^{(n)}_*\mathrm{Res}^{\mathrm{GL}_{nm}(\mathbb{C})}_{\mathfrak{S}_n(\mathrm{U}(m))} V^{\lambda} \right)$$ where $\eta^{(n)}_*$ denotes the representation-theoretic pullback along $\eta^{(n)}:\mathfrak{S}_n(G) \rightarrow \mathfrak{S}_n(\mathrm{U}(m)).$
\end{defn} 

In order to algebraically compute the values $d_{\rho,\lambda}^{\eta}$, we introduce the following generating functions.

\begin{defn}
    For $\rho \in \Phi$, define $\Xi_{\rho,\eta}$ as the multi-set of eigenvalues of the characteristic polynomial of any $\eta^{(|\rho|)}(\gamma)$ for $\gamma \in C_{\rho}$. For $f \in \Lambda$, we will write $f(\Xi_{\rho,\eta})$ for the evaluation $f(\alpha_1,\ldots,\alpha_{nm})$ where $\Xi_{\rho,\eta} = \{\alpha_1,\ldots,\alpha_{nm}\}.$ Given $\ell(\lambda) \leq nm$, define $F_{\lambda,\eta}^{(n)}(\fX) \in R$ as 
    $$F_{\lambda,\eta}^{(n)}(\fX):= \sum_{\rho \in \Phi_n} \frac{s_{\lambda}(\Xi_{\rho,\eta})}{Z_{\rho}} P_{\rho}(\fX).$$
\end{defn}

The wreath symmetric functions $F_{\lambda,\eta}^{(n)}(\fX)$ are generating functions for the branching coefficients $d_{\rho,\lambda}^{\eta}$ in the following sense:

\begin{prop}\label{formula for branching coeff}
For $\rho \in \Psi_n$ and $\lambda \in \Par$ with $\ell(\lambda) \leq nm$,
    $$d_{\rho,\lambda}^{\eta} = \left( F_{\lambda,\eta}^{(n)}(\fX), S_{\rho}(\fX) \right).$$
\end{prop}
\begin{proof}
    By \cite{Macdonald} [Appendix B, 7.4, pg. 174] [Appendix B, 9.7, pg. 178], we know that $(S_{\rho},S_{\gamma}) = \delta_{\rho,\gamma}$ and $S_{\rho}(\fX)$ is the Frobenius characteristic of the irreducible representation $W_{\rho}$ of $\mathfrak{S}_n(G).$ By definition, $F_{\lambda,\eta}^{(n)}(\fX)$ is the Frobenius characteristic of $\eta^{(n)}_*\mathrm{Res}^{\mathrm{GL}_{nm}(\mathbb{C})}_{\mathfrak{S}_n(\mathrm{U}(m))} V^{\lambda}$ and therefore, the multiplicity of $W_{\rho}$ in $\eta^{(n)}_*\mathrm{Res}^{\mathrm{GL}_{nm}(\mathbb{C})}_{\mathfrak{S}_n(\mathrm{U}(m))} V^{\lambda}$, i.e., the coefficient of $S_{\rho}(\fX)$ in the $S$-basis expansion of $F_{\lambda,\eta}^{(n)}(\fX)$, is given by $\left( F_{\lambda,\eta}^{(n)}(\fX), S_{\rho}(\fX) \right).$
\end{proof}

In order to find a workable formula for the Frobenius characteristic $F_{\lambda,\eta}^{(n)}(\fX)$, we will first find a formula for the values $s_{\lambda}(\Xi_{\rho,\eta}).$ 

\begin{defn}
    For $\rho \in \Phi$, define 
    $g_{\rho,\eta}(X):= \sum_{\lambda \in \Par} \frac{p_{\lambda}(X)p_{\lambda}(\Xi_{\rho,\eta})}{z_{\lambda}}.$
\end{defn}

The following is an immediate consequence of the definition of the Hall pairing:

\begin{lem}\label{substitution lem}
For all $\rho \in \Phi$ and $f \in \Lambda,$
    $f(\Xi_{\rho,\eta}) = \langle g_{\rho,\eta}(X), f(X) \rangle .$
\end{lem}

It will be necessary to describe the eigenvalues of the matrices $\gamma \in \mathfrak{S}_n(G).$

\begin{defn}
    Given $c \in G^*$ define the polynomial $f_{c,\eta}(t):=  \det(\mathrm{Id}_{m} - t\eta(\gamma))$ where $\gamma \in c.$ More generally, for $\rho \in \Phi$, define the polynomial $f_{\rho,\eta}(t):= \det\left( \mathrm{Id}_{|\rho|m} - t \eta^{(|\rho|)}(\gamma)  \right)$ where $\gamma$ is any element of $C_{\rho}.$ Define $\Omega_{c,\eta}(X):= \prod_{i\geq 1}f_{c,\eta}(x_i)^{-1}.$
\end{defn}

We need the next lemma. Write $0_m$ for the $m \times m$ matrix containing all $0$'s.

\begin{lem}\label{det lem}
    For all $m\times m$ matrices $A_1,\ldots, A_r$ with $A_r$ invertible, 
    $$\det \begin{pmatrix}
        \mathrm{Id}_m & 0_m & 0_m & \cdots & 0_m &0_m &A_r \\
       A_1 & \mathrm{Id}_m & 0_m & \ldots & 0_m & 0_m & 0_m \\
        0_m & A_2 & \mathrm{Id}_m & \ldots & 0_m & 0_m &0_m \\
        \vdots & \vdots & \vdots & \ddots & \vdots & \vdots & \vdots \\
         0_m & 0_m & 0_m&  \cdots & \mathrm{Id}_m &  0_m  & 0_m\\
         0_m & 0_m & 0_m& \cdots & A_{r-2} &  \mathrm{Id}_m  & 0_m\\
        0_m & 0_m & 0_m & \cdots & 0_m & A_{r-1} & \mathrm{Id}_m \\
    \end{pmatrix} = \det\left( \mathrm{Id}_m + (-1)^{r-1}A_r\cdots A_1 \right).$$
\end{lem}
\begin{proof}
    Let $M$ be the matrix 
    $$M:= \begin{pmatrix}
        \mathrm{Id}_m & 0_m & 0_m & \cdots & 0_m &0_m &A_r \\
       A_1 & \mathrm{Id}_m & 0_m & \ldots & 0_m & 0_m & 0_m \\
        0_m & A_2 & \mathrm{Id}_m & \ldots & 0_m & 0_m &0_m \\
        \vdots & \vdots & \vdots & \ddots & \vdots & \vdots & \vdots \\
         0_m & 0_m & 0_m&  \cdots & \mathrm{Id}_m &  0_m  & 0_m\\
         0_m & 0_m & 0_m& \cdots & A_{r-2} &  \mathrm{Id}_m  & 0_m\\
        0_m & 0_m & 0_m & \cdots & 0_m & A_{r-1} & \mathrm{Id}_m \\
    \end{pmatrix}.$$
    
    We may decompose $M$ into blocks as $M = \begin{pmatrix}
        A & B \\
        C & D
    \end{pmatrix}$ where $A = \mathrm{Id}_m$, $B = \begin{pmatrix}
        0_m & \cdots& 0_m &A_r
    \end{pmatrix}$, $C = \begin{pmatrix}
        A_1 \\
        0_m \\
        \vdots \\
        0_m \\
    \end{pmatrix}$, and 
    $$D = \begin{pmatrix}
        
        \mathrm{Id}_m & 0_m & \ldots & 0_m & 0_m & 0_m \\
         A_2 & \mathrm{Id}_m & \ldots & 0_m & 0_m &0_m \\
         \vdots & \vdots & \ddots & \vdots & \vdots & \vdots \\
          0_m & 0_m&  \cdots & \mathrm{Id}_m &  0_m  & 0_m\\
          0_m & 0_m& \cdots & A_{r-2} &  \mathrm{Id}_m  & 0_m\\
         0_m & 0_m & \cdots & 0_m & A_{r-1} & \mathrm{Id}_m \\
    \end{pmatrix}.$$ The Schur complement of $A$ in this decomposition is given by 
    $$M':= D - CA^{-1}B = \begin{pmatrix}
        
        \mathrm{Id}_m & 0_m & \ldots & 0_m & 0_m & -A_1A_r \\
         A_2 & \mathrm{Id}_m & \ldots & 0_m & 0_m &0_m \\
         \vdots & \vdots & \ddots & \vdots & \vdots & \vdots \\
          0_m & 0_m&  \cdots & \mathrm{Id}_m &  0_m  & 0_m\\
          0_m & 0_m& \cdots & A_{r-2} &  \mathrm{Id}_m  & 0_m\\
         0_m & 0_m & \cdots & 0_m & A_{r-1} & \mathrm{Id}_m \\
    \end{pmatrix}$$ and so by Schur's determinant formula, 
    $\det(M) = \det(A)\det(M') = 1\cdot \det(M') = \det(M') .$ Notice that $M'$ has the same block structure as $M$ so we may repeatedly apply this procedure to find that 
    $$\det(M) = \begin{pmatrix}
        \mathrm{Id}_m &(-1)^{r-2}A_{r-2}\cdots A_1A_r \\
        A_{r-1} & \mathrm{Id}_m \\
    \end{pmatrix}.$$ Applying Schur's formula one last time yields 
    $$\begin{pmatrix}
        \mathrm{Id}_m &(-1)^{r-2}A_{r-2}\cdots A_1A_r \\
        A_{r-1} & \mathrm{Id}_m \\
    \end{pmatrix} = \det \left( \mathrm{Id}_m + (-1)^{r-1} A_{r-1}\cdots A_1A_r \right).$$ Since $A_r$ is invertible we may conjugate by $A_r$ to obtain $\det(M) = \det \left( \mathrm{Id}_m + (-1)^{r-1} A_{r-1}\cdots A_1A_r \right) = \det \left( \mathrm{Id}_m + (-1)^{r-1} A_{r}\cdots A_1 \right).$
\end{proof}

The polynomials $f_{\rho,\eta}(t)$ have the following simple description:

\begin{lem}\label{char pol lem}
   For $\rho \in \Phi$,  
    $f_{\rho,\eta}(t) = \prod_{c \in G^*} \prod_{\ell \geq 1} f_{c,\eta}(t^{\ell})^{m_{\ell}(\rho(c))}.$
\end{lem}
\begin{proof}
Let $\rho \in \Phi_n$ and $\gamma = (g_1,\ldots,g_n; \sigma) \in C_{\rho}.$ The matrix $\eta^{(n)}(\gamma)$ is conjugate in $\mathrm{GL}_{nm}(\mathbb{C})$ to a block diagonal matrix with each block of the form $( \eta(g_{i_1}),\ldots, \eta(g_{i_r}); ( 1\cdots r)) \in \mathfrak{S}_r(G)$ where $(i_1,\ldots,i_r)$ range over the cycles of $\sigma$ and $(1\cdots r)$ denotes the standard long cycle in $\mathfrak{S}_r.$ It is straightforward to see by applying Lemma \ref{det lem} with $A_j:= \eta(g_{i_{r-j+1}})$ that
$$\det\left(\mathrm{Id}_{rm} - t ( \eta(g_{i_1}),\ldots, \eta(g_{i_r}); ( 1\cdots r)) \right) = \det( \mathrm{Id}_{m} - t^r \eta(g_{i_1}\cdots g_{i_r}))= f_{c,\eta}(t^r)$$ where $c$ is the conjugacy class of $g_{i_1}\cdots g_{i_r}.$ The result follows.
\end{proof}

The series $g_{\rho,\eta}(X)$ has an explicit product formula.

\begin{prop}\label{sub kernel prop}
For all $\rho \in \Phi,$
    $g_{\rho,\eta}(X) = \prod_{i \geq 1} \prod_{c \in G^*} \prod_{\ell \geq 1} f_{c,\eta}(x_i^{\ell})^{-m_{\ell}(\rho(c))}.$
\end{prop}
\begin{proof}
 In what follows, let $\rho \in \Phi$ and $\gamma \in C_{\rho}$. By definition, $g_{\rho,\eta}(X) = \sum_{\lambda \in \Par} \frac{p_{\lambda}(X)p_{\lambda}(\Xi_{\rho,\eta})}{z_{\lambda}}$ which we may rewrite as $\exp\left( \sum_{r \geq 1} \frac{p_{r}(X)p_{r}(\Xi_{\rho,\eta})}{r} \right).$ Expanding $p_r(X) = \sum_{i\geq 1}x_i^r$ yields
    $$\prod_{i \geq 1}\exp\left( \sum_{r \geq 1} \frac{x_i^{r}p_{r}(\Xi_{\rho,\eta})}{r}  \right).$$ Now we notice that $p_r(\Xi_{\rho,\eta}) = \mathrm{Tr}(\eta(\gamma)^r)$ and furthermore, for all $i \geq 1,$
    $$\exp\left( \sum_{r \geq 1} \frac{x_i^{r}p_{r}(\Xi_{\rho,\eta})}{r}  \right) = \det(1-x_i \eta(\gamma))^{-1}.$$ Therefore, 
    $$\prod_{i \geq 1}\exp\left( \sum_{r \geq 1} \frac{x_i^{r}p_{r}(\Xi_{\rho,\eta})}{r}  \right) = \prod_{i \geq 1} \det(1-x_i \eta(\gamma))^{-1} = \prod_{i \geq 1} f_{\rho,\eta}(x_i)^{-1}.$$ Lastly, by Lemma \ref{char pol lem} we find that 
    $$\prod_{i \geq 1} f_{\rho,\eta}(x_i)^{-1} = \prod_{i \geq 1} \prod_{c \in G^*} \prod_{\ell \geq 1} f_{c,\eta}(x_i^{\ell})^{-m_{\ell}(\rho(c))}.$$
\end{proof}

Here we show that the generating series for the $g_{\rho,\eta}(Y)$ series has a useful plethystic form.

\begin{prop}\label{plethystic kernel formula prop}

    $$\sum_{\rho \in \Phi} \frac{P_{\rho}(\fX)g_{\rho,\eta}(Y)}{Z_{\rho}} = \Omega\left( \sum_{c \in G^*} \zeta_c^{-1} X_{c} \Omega_{c,\eta}(Y) \right)$$
\end{prop}
\begin{proof}
    Using Proposition \ref{sub kernel prop}, 
        $$\sum_{\rho \in \Phi} \frac{P_{\rho}(\fX)g_{\rho,\eta}(Y)}{Z_{\rho}} = \sum_{\rho \in \Phi} \frac{P_{\rho}(\fX)}{Z_{\rho}} \prod_{i \geq 1} \prod_{c \in G^*} \prod_{\ell \geq 1} f_{c,\eta}(y_i^{\ell})^{-m_{\ell}(\rho(c))}.$$ We may factor each of these summands as 
        $$\frac{P_{\rho}(\fX)}{Z_{\rho}} \prod_{i \geq 1} \prod_{c \in G^*} \prod_{\ell \geq 1}f_{c,\eta}(y_i^{\ell})^{-m_{\ell}(\rho(c))} = \prod_{c \in G^*} \left( \frac{p_{\rho(c)}(X_{c})}{z_{\rho(c)} \zeta_c^{\ell(\rho(c))}} \prod_{\ell \geq 1} \prod_{i \geq 1} f_{c,\eta}(y_i^{\ell})^{-m_{\ell}(\rho(c))} \right)$$ which we will write using plethystic substitution as 
        $$\prod_{c \in G^*} \left( \frac{p_{\rho(c)}(X_{c})}{z_{\rho(c)} \zeta_c^{\ell(\rho(c))}} \prod_{\ell \geq 1} \prod_{i \geq 1} f_{c,\eta}(y_i^{\ell})^{-m_{\ell}(\rho(c))} \right) = \prod_{c \in G^*} \left( \frac{p_{\rho(c)}\left(X_{c} \prod_{i \geq 1} f_{c,\eta}(y_i)^{-1} \right)}{z_{\rho(c)} \zeta_c^{\ell(\rho(c))}}  \right).$$ Note $\prod_{i \geq 1} f_{c,\eta}(y_i)^{-1} = \Omega_{c,\eta}(Y).$

        Therefore, 
        $$\sum_{\rho \in \Phi} \frac{P_{\rho}(\fX)}{Z_{\rho}} \prod_{i \geq 1} \prod_{c \in G^*} \prod_{\ell \geq 1} f_{c,\eta}(y_i^{\ell})^{-m_{\ell}(\rho(\zeta))} = \prod_{c\in G^*}\left( \sum_{\lambda \in \Par} \frac{p_{\lambda}\left(X_{\zeta} \Omega_{c,\eta}(Y) \right)}{z_{\lambda} \zeta_c^{\ell(\lambda)}}\right)$$ which simplifies to 
        $$ \prod_{c \in G^*} \Omega \left( \zeta_c^{-1} X_{c} \Omega_{c,\eta}(Y) \right) = \Omega\left( \sum_{c \in G^*} \zeta_c^{-1} X_{c} \Omega_{c,\eta}(Y) \right).$$
\end{proof}

It will be convenient to introduce the following notation:

\begin{defn}
    For $\lambda \in \Par,$ define $F_{\lambda,\eta}(\fX):= \sum_{\rho \in \Phi} \frac{P_{\rho}(\fX) s_{\lambda}(\Xi_{\rho,\eta})}{Z_{\rho}}.$
\end{defn}

Note that for $\ell(\lambda) \leq nm,$ the homogeneous degree $n$ component of $F_{\lambda,\eta}(\fX)$ is $F^{(n)}_{\lambda,\eta}(\fX).$ We may now give a useful formula for the series $F_{\lambda,\eta}(\fX).$

\begin{prop}\label{restriction character formula prop}
For $\lambda \in \Par,$
    $F_{\lambda,\eta}(\fX) = \left\langle \Omega\left( \sum_{c \in G^*} \zeta_c^{-1} X_{c} \Omega_{c,\eta}(Y) \right), s_{\lambda}(Y) \right\rangle .$
\end{prop}
\begin{proof}
Using Lemma \ref{substitution lem}, we see that
$$F_{\lambda,\eta}(\fX) = \sum_{\rho \in \Phi} \frac{s_{\lambda}(\Xi_{\rho,\eta})}{Z_{\rho}} P_{\rho}(\fX) = \sum_{\rho \in \Phi} \frac{\langle g_{\rho,\eta}(Y), s_{\lambda}(Y) \rangle}{Z_{\rho}} P_{\rho}(\fX)$$ which we may write as 
$F_{\lambda,\eta}(\fX) =  \left\langle \sum_{\rho \in \Phi} \frac{P_{\rho}(\fX)g_{\rho,\eta}(Y)}{Z_{\rho}} , s_{\lambda}(Y) \right\rangle.$ From Proposition \ref{plethystic kernel formula prop}, we find
$$\left\langle \sum_{\rho \in \Phi} \frac{P_{\rho}(\fX)g_{\rho,\eta}(Y)}{Z_{\rho}} , s_{\lambda}(Y) \right\rangle = \left\langle \Omega\left( \sum_{c \in G^*} \zeta_c^{-1} X_{c} \Omega_{c,\eta}(Y) \right) , s_{\lambda}(Y) \right\rangle.$$

\end{proof}

We introduce the following notation:

\begin{defn}\label{G def}
For $\rho \in \Psi$, define
    $G_{\rho,\eta}(Y):= \left( \Omega\left( \sum_{c \in G^*} \zeta_c^{-1} X_{c} \Omega_{c,\eta}(Y)   \right) , S_{\rho}(\fX) \right).$
\end{defn}

The utility of defining the series $G_{\rho,\eta}(Y)$ is demonstrated by the next formula.

\begin{lem}\label{branching coeff lem 2}
For all $\rho \in \Psi_n$ and $\lambda \in \Par$ with $\ell(\lambda) \leq nm$,
    $d_{\rho,\lambda}^{\eta} = \langle G_{\rho,\eta}(Y), s_{\lambda}(Y) \rangle.$
\end{lem}
\begin{proof}
    From Proposition \ref{formula for branching coeff}, we know that 
    $d_{\rho,\lambda}^{\eta} = \left( F_{\lambda,\eta}^{(n)}(\fX), S_{\rho}(\fX) \right).$ Note that since $|\rho| = n,$ we may exchange $F_{\lambda,\eta}^{(n)}(\fX)$ with $F_{\lambda,\eta}(\fX)$ in the above to find 
    $d_{\rho,\lambda}^{\eta} = \left( F_{\lambda,\eta}(\fX), S_{\rho}(\fX) \right).$ Proposition \ref{restriction character formula prop} now shows 
    $$\left( F_{\lambda,\eta}(\fX), S_{\rho}(\fX) \right) = \left( \left\langle \Omega\left( \sum_{c \in G^*} \zeta_c^{-1} X_{c} \Omega_{c,\eta}(Y)   \right), s_{\lambda}(Y) \right\rangle, S_{\rho}(\fX) \right).$$ Swapping the order of the pairing evaluations yields
    $$\left( \left\langle \Omega\left( \sum_{c \in G^*} \zeta_c^{-1} X_{c} \Omega_{c,\eta}(Y)   \right), s_{\lambda}(Y) \right\rangle, S_{\rho}(\fX) \right) = \left\langle \left( \Omega\left( \sum_{c \in G^*} \zeta_c^{-1} X_{c} \Omega_{c,\eta}(Y)   \right), S_{\rho}(\fX)  \right), s_{\lambda}(Y) \right\rangle$$ and thus
    $ d_{\rho,\lambda}^{\eta} = \left\langle G_{\rho,\eta}(Y) , s_{\lambda}(Y)   \right\rangle.$
\end{proof}

In what follows, let $\mathbb{S}_{\mu}(\eta)$ denote the \textbf{\textit{Schur functor}} for $\mu \in \Par$ applied to the representation $\eta$ of $G.$ Explicitly, 
$\mathbb{S}_{\mu}(V) = S^{\mu} \otimes_{\mathfrak{S}_{|\mu|}} V^{\otimes |\mu|}$ for all finite dimensional $\mathbb{C}$ vector spaces $V$ and we define $\mathbb{S}_{\mu}(\eta)$ accordingly.

\begin{lem}\label{irrep alphabet tranform lem}
For all $\gamma \in G_*,$
    $$\overline{\varphi_{\gamma}}\left( \Omega_{c,\eta}(Y) | c \in G_* \right) =  \sum_{\mu \in \Par} \dim \Hom_{G}(\gamma, \mathbb{S}_{\mu}(\eta)) s_{\mu}(Y).$$
\end{lem}
\begin{proof}
    By direct computation, we find
    \begin{align*}
        \overline{\varphi_{\gamma}}\left( \Omega_{c,\eta}(Y) | c \in G_* \right) &= \sum_{c \in G_*} \zeta_c^{-1} \overline{\chi^{\gamma}(c)} \Omega_{c,\eta}(Y)\\
        &= \sum_{c \in G^*} \zeta_c^{-1} \overline{\chi^{\gamma}(c)} \sum_{\mu \in \Par} s_{\mu}(\Xi_{c,\eta} ) s_{\mu}(Y) \\
        &= \sum_{\mu \in \Par} s_{\mu}(Y) \sum_{c \in G^*} \zeta_c^{-1} \overline{\chi^{\gamma}(c)} s_{\mu}(\Xi_{c,\eta})\\
    \end{align*}
and by Schur-Weyl duality, $s_{\mu}(\Xi_{c,\eta})$ is the character value of the representation $\mathbb{S}_{\mu}(\eta)$ at the conjugacy class $c \in G^*.$ Importantly, we have the character orthogonality relation $$ \sum_{c \in G^*} \zeta_c^{-1} \overline{\chi(c)} \psi(c) = \begin{cases}
    1 & \chi \cong \psi\\
    0 & \chi \ncong \psi \\
\end{cases}$$ for all irreducible characters $\chi$ and $\psi$ of $G$ and thus
$$\sum_{c \in G^*} \zeta_c^{-1} \overline{\chi^{\gamma}(c)} s_{\mu}(\Xi_{c,\eta}) = \dim \Hom_{G}(\gamma, \mathbb{S}_{\mu}(\eta)).$$ The result follows.
\end{proof}

We require the reproducing kernel property of the pairing $(-,-).$

\begin{lem}\label{wreath kernel lem}
For any collection of variable sets $\mathfrak{Y} = (Y_{c})_{c \in G^*}$ and $f \in R,$
    $$\left( \Omega \left( \sum_{c \in G^*} \zeta_c^{-1} X_{c}Y_{c} \right), f(\fX) \right) = \overline{f}(\mathfrak{Y}).$$
\end{lem}
\begin{proof}
    From Macdonald \cite{Macdonald} [Appendix B, 7.4, pg. 174], we know
    $$\Omega \left( \sum_{c \in G^*} \zeta_c^{-1} X_{c}Y_{c} \right) = \sum_{\rho \in \Psi} S_{\rho}(\fX) \overline{S_{\rho}}(\mathfrak{Y})$$ and $(S_{\rho}(\fX),S_{\gamma}(\mathfrak{X})) = \delta_{\rho,\gamma}$ for all $\rho,\gamma \in \Psi.$ If we assume $f(\fX) = \sum_{\rho} a_{\rho} S_{\rho}(\fX)$, then 
    $$\left( \Omega \left( \sum_{c\in G^*} \zeta_c^{-1} X_{c}Y_{c} \right), f(\fX) \right) = \sum_{\rho} \overline{a_{\rho}}\left( \Omega \left( \sum_{c \in G^*} \zeta_c^{-1} X_{c}Y_{c} \right),  S_{\rho}(\fX) \right) = \sum_{\rho} \overline{a_{\rho}} \overline{S_{\rho}}(\mathfrak{Y}) = \overline{f}(\mathfrak{Y}).$$ 
\end{proof}

Using the above reproducing kernel property, we find an explicit formula for the series $G_{\rho,\eta}(Y).$ 

\begin{prop}\label{plethystic formula for irrep sub prop}
For all $\rho \in \Psi,$
    $G_{\rho,\eta}(Y) = \prod_{\gamma \in G_*} s_{\rho(\gamma)}\left(\sum_{\mu \in \Par} \dim \Hom_{G}(\gamma, \mathbb{S}_{\mu}(\eta)) s_{\mu}(Y)\right)$
\end{prop}
\begin{proof}
    Applying Lemma \ref{wreath kernel lem} to Definition \ref{G def} shows
    \begin{align*}
        G_{\rho,\eta}(Y) &= \left( \Omega\left( \sum_{c \in G^*} \zeta_c^{-1} X_{\zeta} \Omega_{c,\eta}(Y)   \right) , S_{\rho}(\fX) \right) \\
        &= \left( \Omega\left( \sum_{c \in G^*} \zeta_c^{-1} X_{c} \Omega_{c,\eta}(Y)   \right) , \prod_{\gamma \in G_*} s_{\rho(\gamma)}(\varphi_{\gamma}(\fX)) \right)\\
        &= \prod_{\gamma\in G_*} s_{\rho(\gamma)}\left(\overline{\varphi_{\gamma}}\left( \Omega_{c,\eta}(Y) | c \in G_* \right)\right)\\
        &= \prod_{\gamma \in G_*} s_{\rho(\gamma)}\left(\sum_{\mu \in \Par} \dim \Hom_{G}(\gamma, \mathbb{S}_{\mu}(\eta)) s_{\mu}(Y)\right).\\
    \end{align*}
    
\end{proof}

By combining Lemma \ref{branching coeff lem 2} and Proposition \ref{plethystic formula for irrep sub prop} we find the following:

\begin{thm}\label{main theorem}
For $\rho \in \Psi_n$ and $\lambda \in \Par$ with $\ell(\lambda) \leq n m ,$
    $$\mathrm{dim}_{\mathbb{C}} \mathrm{Hom}_{\mathfrak{S}_{n}(G)}\left( W_{\rho}, \eta^{(n)}_*\mathrm{Res}^{\mathrm{GL}_{nm}(\mathbb{C})}_{\mathfrak{S}_n(\mathrm{U}(m))} V^{\lambda} \right)= \left\langle \prod_{\gamma \in G_*} s_{\rho(\gamma)}\left(\sum_{\mu \in \Par} \dim \Hom_{G}(\gamma, \mathbb{S}_{\mu}(\eta)) s_{\mu}\right), ~ s_{\lambda} \right\rangle.$$
\end{thm}

\subsection{Cyclic case}

Let $m \geq 1$ and fix $\zeta_m:= e^{\frac{2 \pi i}{m}}$ to be a primitive complex $m$-th root of unity.
Here consider the case $G = \mu_m = \{\zeta_m^{j} \mid 0 \leq j \leq m-1\}$, the group of complex $m$-th roots of unity. Take $\eta:\mu_m \rightarrow U(1)$ to be the identity homomorphism $\eta = \mathrm{Id}$. In this case, for all $n \geq 1$ and $\ell(\lambda) \leq n,$ $\eta_{*}^{(n)} \mathrm{Res}_{\mathfrak{S}_n(U(1))}^{\mathrm{GL}_{n}(\mathbb{C})} V^{\lambda} = \mathrm{Res}_{\mathfrak{S}_n(\mu_m)}^{\mathrm{GL}_{n}(\mathbb{C})} V^{\lambda}.$ The irreducible representations $\rho$ of $\mathfrak{S}_n(\mu_m)$ are indexed by maps $\rho: \mu_m \rightarrow \Par$ with $|\rho| = \sum_{0 \leq j \leq m-1} | \rho(\zeta_m^{j})| = n.$ For $0 \leq j \leq m-1,$ let $\gamma_j : \mu_m \rightarrow U(1)$ denote the character $\gamma_j(\zeta):= \zeta^{j}$ for $\zeta \in \mu_m.$ As a consequence of Theorem \ref{main theorem}, we find the following. 

\begin{cor}\label{cyclic cor}
For $\rho: \mu_m \rightarrow \Par$ and $\lambda \in \Par$ with $\ell(\lambda) \leq |\rho| = n,$
    $$\mathrm{dim}_{\mathbb{C}} \mathrm{Hom}_{\mathfrak{S}_{n}(\mu_m)}\left( W_{\rho}, \mathrm{Res}^{\mathrm{GL}_n(\mathbb{C})}_{\mathfrak{S}_n(\mu_m)} V^{\lambda} \right) = \left\langle \prod_{0\leq j \leq m-1} s_{\rho(\zeta_m^j)}\left( \sum_{k \geq 0} h_{km+j}\right), ~ s_{\lambda} \right\rangle.$$
\end{cor}
\begin{proof}
    Theorem \ref{main theorem} tells us that 
    $$\mathrm{dim}_{\mathbb{C}} \mathrm{Hom}_{\mathfrak{S}_{n}(\mu_m)}\left( W_{\rho}, \mathrm{Res}^{\mathrm{GL}_n(\mathbb{C})}_{\mathfrak{S}_n(\mu_m)} V^{\lambda} \right) = \left\langle \prod_{0 \leq j \leq m-1} s_{\rho(\zeta_m^j)}\left(\sum_{\mu \in \Par} \dim \Hom_{\mu_m}(\gamma_j, \mathbb{S}_{\mu}(\gamma_1)) s_{\mu}\right), ~ s_{\lambda} \right\rangle.$$ Since $\gamma_1$ is $1$-dimensional, $\mathbb{S}_{\mu}(\gamma_1) = 0$ for $\ell(\mu) > 1$ and otherwise, $\mathbb{S}_{(n)}(\gamma_1) = \gamma_1^{\otimes n} = \gamma_a$ for $n \geq 0$ where $a$ is the residue of $n$ modulo $m.$ Therefore, $$\dim \Hom_{\mu_m}(\gamma_j, \mathbb{S}_{\mu}(\gamma_1)) = \begin{cases}
        1 & \mu = (n), m \mid n-j\\
        0 & \text{otherwise}\\
    \end{cases}$$ and as such 
    $$\sum_{\mu \in \Par} \dim \Hom_{\mu_m}(\gamma_j, \mathbb{S}_{\mu}(\gamma_1)) s_{\mu} = \sum_{k \geq 0} h_{km+j}.$$ Thus, the formula holds.
\end{proof}

\printbibliography

\end{document}